\documentclass[12pt,reqno]{amsart}

\usepackage{amssymb}
\usepackage[all]{xy}

\newtheorem{theorem}{Theorem}[section]

\newtheorem{lemma}[theorem]{Lemma}

\theoremstyle{definition}

\numberwithin{equation}{section}

\begin{document}

\baselineskip=15.5pt

\title[Turbulent foliations on complex tori and transversely Cartan geometry]{Turbulent
holomorphic foliations on compact complex tori and transversely holomorphic Cartan geometry}

\author[I. Biswas]{Indranil Biswas}

\address{Department of Mathematics, Shiv Nadar University, NH91, Tehsil Dadri,
Greater Noida, Uttar Pradesh 201314, India}

\email{indranil.biswas@snu.edu.in, indranil29@gmail.com}

\author[S. Dumitrescu]{Sorin Dumitrescu}

\address{Universit\'e C\^ote d'Azur, CNRS, LJAD, France}

\email{dumitres@unice.fr}

\subjclass[2020]{32S65, 53C10, 53C30}

\keywords{Holomorphic foliation, compact complex torus,  transversely complex projective structure, Cartan geometry,  transversely Cartan geometry}

\date{}

\begin{abstract}
We define a class of nonsingular holomorphic foliations on compact complex tori which generalizes (in higher codimension) 
the turbulent foliations of codimension one constructed by Ghys in \cite{Gh}. For those smooth turbulent foliations we prove that all transversely holomorphic Cartan geometries are flat. We also  establish  a uniqueness result for the transversely holomorphic Cartan geometries.

\end{abstract}

\maketitle

\section{Introduction}

This article studies transversely holomorphic Cartan geometries on nonsingular holomorphic foliations on compact complex 
tori (i.e., compact complex manifolds which are biholomorphic to a quotient ${\mathbb C}^d / \Lambda$, with $\Lambda$ being 
a normal lattice in ${\mathbb C}^d$, with $d\,\geq\, 2$).

A first motivation of our study comes from the following classification result for nonsingular codimension one holomorphic 
foliations on compact complex tori, proved by Ghys in \cite{Gh}: Any nonsingular {\it codimension one} holomorphic foliation 
on a compact complex torus is either {\it linear} (meaning it is defined as the kernel of a global nonzero holomorphic 
one-form on the torus) or it is {\it turbulent}.

A codimension one holomorphic foliation on a compact complex torus $A$ is {\it turbulent} in Ghys' sense if it is defined by 
the kernel of a closed {\it meromorphic} one-form constructed as follows.  Assume that there exists a holomorphic submersion 
$\pi \,:\, A \,\longrightarrow\, T$, over an elliptic curve $T$ and consider the closed meromorphic one-form $\eta \,=\, 
(\pi^*\omega) + \beta$, where $\omega$ is a meromorphic one-form on $T$ and $\beta$ is a (translation invariant) holomorphic 
one-form on the torus $A$ that does not vanish on the fibers of the fibration $\pi$. The kernel of $\eta$ defines (away from 
the polar divisor of $\eta$) a holomorphic codimension one foliation which extends to a nonsingular holomorphic foliation on 
the torus $A$ (see \cite{Gh}). Notice that in the case where the holomorphic one-form $\beta$ vanishes on the fibers of 
$\pi$, the turbulent foliation degenerates to the fibration $\pi$.

The dynamics of the turbulent foliation was described in \cite{Gh} (see also \cite{Br2}): the inverse image through the 
fibration $\pi$ of the polar divisor of the meromorphic one-form $\omega$ is a finite union of compact leafs and all other 
leafs are noncompact and accumulate on every compact leaf.

Codimension one linear foliations on compact complex tori --- being defined by a holomorphic one-form $\beta$ (which is 
automatically closed and translation invariant) --- are necessarily translation invariant and admit a transversely complex 
translation structure. The local (transverse) charts of the corresponding transversely complex translation structure are 
given by the local primitives of $\beta$ (seen as local submersions $f_i \,:\, U_i \,\longrightarrow \,\mathbb C$
that locally define the foliation): the transition maps between two such local charts is necessarily given by the
post-composition with a 
translation.  More precisely, for each connected component of a nonempty intersection $U_i \cap U_j$ there exists a constant 
$c_{ij}$ such that $f_i\,=\,f_j+c_{ij}$.

Recall that Lie's classification of dimension one homogeneous spaces implies that the possible transverse structures for a 
codimension one foliation are exactly of the following three types: translation structures, affine structures and projective 
structures. Codimension one holomorphic foliations which are transversely complex affine or transversely complex projective 
were studied by several authors (see, \cite{Se, Sc, LP}).

In this context it was proved in \cite{BD3} that there exists codimension one turbulent foliations on compact complex tori 
which do not admit any transversely complex projective structure. More precisely, the main result in \cite{BD3} asserts that 
on the product of two elliptic curves, a generic smooth (nonsingular) turbulent foliation does not admit any transversely complex projective 
structure \cite{BD3}. Indeed, for a polar part of the above meromorphic one-form $\omega$ of degree $ \,\geq\, 8$, generic 
turbulent foliations do not admit any transversely complex projective structure \cite{BD3}. The proof of \cite{BD3} employed 
a dimension counting argument, based on a proposition proving that smooth turbulent codimension one foliations admit at most 
one transversely complex projective structure.  To the best of our knowledge these were the first known examples of 
nonsingular codimension one holomorphic foliations on a complex projective manifold admitting no transversely complex 
projective structures. Recall also that all smooth fibrations (submersions) over a Riemann surface admit a nonsingular 
transversely projective structure which is a consequence of the uniformization theorem for Riemann surfaces applied to the 
base of the fibration (which happens to be here a global transversal of the foliation).

It should be emphasized that those results in \cite{BD3} deal with {\it regular} (meaning {\it nonsingular}) complex 
projective structures. Indeed, codimension one turbulent foliations on complex tori, being defined by global closed 
meromorphic one-forms $\eta$, admit a {\it singular} transversely complex translation structure, which in turn induces 
a singular transversely complex projective structure in the sense of \cite{LPT}.  Moreover, a consequence of 
Brunella's classification of nonsingular holomorphic foliations on compact complex surfaces, \cite{Br1},
says that all those  foliations admit a singular transversely complex projective structure.

A recent result of Fazoli, Melo and Pereira, \cite[Theorem 4.4]{FMP}, classifies all smooth
(nonsingular) turbulent foliations on K\"ahler 
surfaces which do admit a regular (nonsingular) complex projective structure.

Our article here deals with smooth (nonsingular) holomorphic turbulent foliations, with arbitrary codimension, on compact 
complex tori (this notion is defined in Section \ref{section turbulent}). Our main result is that for such a smooth 
turbulent foliation, all transversely holomorphic Cartan geometries are {\it flat} (see Theorem \ref{thm1}).  Recall that 
this means that the foliation is transversely modelled on a complex homogeneous space $G/H$, with $G$ a complex Lie group 
and $H$ a closed subgroup in $G$.  It is equivalent to say that in an adapted holomorphic foliated atlas, the foliation is 
locally given by the level sets of local submersions to the complex homogeneous space $G/H$ which differ on the overlaps of 
two open sets in the atlas by the post-composition with an element in $G$ acting on $G/H$.  The one dimensional homogeneous 
space ${\rm P^1}(\mathbb C)$ acted on by $G\,=\,{\rm {\rm PSL}}(2, \mathbb C)$ (and with $H$ being a maximal parabolic subgroup 
in $G$) corresponds to transversely complex projective structures (for foliations of codimension one). It should be 
emphasized that the codimension one is particular since all Cartan geometry are flat in dimension one and, consequently,  all 
transversely Cartan geometries are automatically flat for codimension one foliations. In general, this is not the case 
anymore in higher codimension.

Given a holomorphic principal $H$--bundle $E_H$ over a compact complex torus $A$  equipped with a flat partial connection in the direction of a smooth holomorphic  turbulent foliation $\mathcal F$, we also prove in Theorem \ref{thm1} 
that for any given homogeneous model space $G/H$, the foliated torus $(A, \mathcal F)$  admits at most 
one compatible transversely holomorphic $G/H$-structure. This generalizes to turbulent foliations --- of arbitrary codimension and to any 
transversely holomorphic Cartan geometry --- the uniqueness result proved in \cite{BD3} for codimension one foliations and 
transversely complex projective structures.

Consider, for example, on a compact complex torus $A$ a codimension one smooth  holomorphic turbulent foliation $\mathcal F$ which 
admits a transversely complex projective structure (for tori of complex dimension two, those foliations were completely 
characterized in \cite[Theorem 4.4]{FMP}). Then the product $A \times A$ inherits a codimension two holomorphic foliation 
given by $\mathcal F \oplus \mathcal F$ which is transversely modelled on ${\rm P^1} (\mathbb C) \times {\rm P^1} (\mathbb 
C)$. As a particular case of our Theorem \ref{thm1} this transversely ${\rm P^1} (\mathbb C) \times {\rm P^1} (\mathbb 
C)$-structure is unique (up to the choice of a natural principal bundle endowed with a flat partial connection along the foliation). Moreover, any transversely holomorphic Cartan geometry with infinitesimal model $G={\rm 
PSL}(2,\mathbb C) \times {\rm PSL}(2,\mathbb C)$ and $H= B \times B$, where $B$ is the maximal parabolic subgroup in ${\rm PSL}(2,\mathbb C),$ is 
flat and defines a transversely ${\rm P^1} (\mathbb C) \times {\rm P^1} (\mathbb C)$-structure.

Moreover, Theorem \ref{thm1} stands true also for branched Cartan geometries. Therefore one is allowed to apply Theorem 
\ref{thm1} to transversely Cartan geometries having mild singularities (of {\it branched type}): the notion of {\it branched 
Cartan geometry} was introduced and studied in \cite{BD} (see also the survey \cite{BD2}).

A natural question left unsolved is to classify all smooth turbulent foliations on compact complex tori that admit a 
transversely holomorphic (or branched) Cartan geometry. Theorem \ref{thm1} proves that the transversely Cartan geometry 
must be flat.

Let us mention some  related works studying turbulent foliations. A more general notion of a possible  singular turbulent foliation was defined and studied  in \cite{Br2}, where Brunella 
classified all (possibly singular) codimension one  foliations on compact complex tori (see also \cite{PS} for a study of  the space of all possible singular turbulent foliations  having a fixed tangency divisor with respect to  a given elliptic fibration). A general study of smooth  holomorphic foliations on compact homogeneous K\"ahler manifolds was carried out in \cite{LoP}.

The organization of this article  is as follows. Section \ref{section turbulent}  introduces a new class of subbundles (of arbitrary rank)  in the holomorphic  tangent bundle $TA$ of a compact complex torus $A$. Those, so-called, {\it generating} subbundles are in some sense  the opposite of the translation invariant subbundles. In particular, the corank one integrable  generating   subbundles in $TA$  coincide with  the tangent spaces of Ghys turbulent foliations. We  prove vanishing properties for generating subbundles  (see Lemma \ref{lem1}) which will be  applied in the sequel  to the tangent spaces of smooth turbulent foliations. Section \ref{section foliation} defines a smooth turbulent foliation on a compact complex torus as determined by an integrable generating subbudle in $TA$. We present the framework of transversely Cartan geometry and transversely  branched Cartan geometry as introduced and developed in \cite{BD,BD2}.  Section \ref{section rigidity} contains the proof of the main result, Theorem \ref{thm1}.

\section{Generating distributions on a torus} \label{section turbulent}

Let $A$ be a compact complex torus of complex dimension $d$. Denote by $TA$ the
holomorphic tangent bundle of $A$. Consider the trivial holomorphic
vector bundle $A\times H^0(A,\, TA)\, \longrightarrow\, A$ with fiber $H^0(A,\, TA)$. Let
\begin{equation}\label{e1}
\Phi\ :\ A\times H^0(A,\, TA)\ \longrightarrow\ TA
\end{equation}
be the evaluation map that sends any $(x,\, s)\, \in\, A\times H^0(A,\, TA)$ to $s(x)\, \in\, T_xA$.
This homomorphism $\Phi$ is an isomorphism, because $TA$ is holomorphically trivial.

Take a holomorphic subbundle
\begin{equation}\label{e2}
{\mathcal F}\ \subset\ TA
\end{equation}
of rank $r$ with $1\,\leq\, r\, \leq\, d-1$. So $\Phi^{-1}({\mathcal F})\ \subset\ A\times H^0(A,\, TA)$
is a holomorphic subbundle, where $\Phi$ is the isomorphism in \eqref{e1}. For each $x\, \in\, A$, consider
the subspace $$\Phi^{-1}({\mathcal F})_x \, \subset\, (A\times H^0(A,\, TA))_x \,=\, H^0(A,\, TA).$$
Let
\begin{equation}\label{e3}
G({\mathcal F})\ := \ \text{Span}\, \{\Phi^{-1}({\mathcal F})_x\}_{x\in A}\ \subset\
H^0(A,\, TA)
\end{equation}
be the subspace of $H^0(A,\, TA)$ generated by its subspaces $\{\Phi^{-1}({\mathcal F})_x\}_{x\in A}$.

The subbundle $\mathcal F$ in \eqref{e2} will be called \textit{generating} if
$G({\mathcal F})\ =\ H^0(A,\, TA)$, where $G({\mathcal F})$ is constructed in \eqref{e3}.

For a subbundle ${\mathcal F}\ \subset\ TA$ as in \eqref{e2}, let
\begin{equation}\label{e4}
{\mathcal N}\ :=\ (TA)/{\mathcal F}
\end{equation}
be the quotient bundle. Let
\begin{equation}\label{e4b}
q\ :\ TA \ \longrightarrow\ (TA)/{\mathcal F}\ =: \ {\mathcal N}
\end{equation}
be the quotient map.

Fix a K\"ahler form $\omega_A$ on $A$ in order to define the notion of  {\it degree}  for  torsionfree coherent
analytic sheaves on $A$. So, for a torsionfree coherent
analytic sheaf $E$ on $A$,
$$
\text{degree}(E)\ :=\ \int_A c_1(E)\wedge\omega^{d-1}_A \ \in\ {\mathbb R}.
$$
For a holomorphic vector bundle $V$ on $A$, let
$$
0\ =\ V_0\ \subset\ V_1\ \subset\ \cdots \ \subset\ V_{\ell-1} \ \subset\ V_\ell\ =\ V
$$
be the Harder--Narasimhan filtration of $V$ (see \cite[p.~16, Theorem 1.3.4]{HL}
for Harder--Narasimhan filtration). Then
\begin{equation}\label{e5}
\mu_{\rm max}(V)\ :=\ \frac{\text{degree}(V_1)}{\text{rank}(V_1)},\ \ \ \,
\mu_{\rm min}(V)\ :=\ \frac{\text{degree}(V/V_{\ell-1})}{\text{rank}(V_/V_{\ell-1})}.
\end{equation}

\begin{lemma}\label{lem1}
Let $\mathcal F$ as in \eqref{e2} be a generating subbundle. Let $V$ be a holomorphic
vector bundle on $A$ with $\mu_{\rm max}(V)\, \leq\, 0$ (see \eqref{e5}). Then
$$
H^0(A,\, V\otimes \bigwedge\nolimits^j{\mathcal N}^*)\ =\ 0
$$
for all $j\, \geq\, 1$, where $\mathcal N$ is constructed in \eqref{e4}.
\end{lemma}

\begin{proof}
We will first prove that
\begin{equation}\label{e6}
\mu_{\rm min}({\mathcal N})\ > \ 0
\end{equation} 
(see \eqref{e5}). Let
$$
0\ =\ {\mathcal N}_0\ \subset\ {\mathcal N}_1\ \subset\ \cdots \ \subset\ {\mathcal N}_{m-1} \ \subset\
{\mathcal N}_m\ =\ {\mathcal N}
$$
be the Harder--Narasimhan filtration of $\mathcal N$. Since ${\mathcal N}$ is a quotient of
$TA$ (see \eqref{e4b}), and ${\mathcal N}/{\mathcal N}_{m-1}$ is a quotient of $\mathcal N$, it follows
that ${\mathcal N}/{\mathcal N}_{m-1}$ is a quotient of $TA$. As $TA$ is semistable of degree zero,
its quotient ${\mathcal N}/{\mathcal N}_{m-1}$ has the following property:
\begin{equation}\label{e-1}
\mu_{\rm min}({\mathcal N}) \ =\
\frac{\text{degree}({\mathcal N}/{\mathcal N}_{m-1})}{\text{rank}({\mathcal N}/{\mathcal N}_{m-1})}
 \ \geq\ 0.
\end{equation}

In view of \eqref{e-1}, to prove \eqref{e6} by contradiction, assume the following:
\begin{equation}\label{e6b}
\text{degree}({\mathcal N}/{\mathcal N}_{m-1}) \ = \ 0.
\end{equation}
If the rank of ${\mathcal N}/{\mathcal N}_{m-1}$ is $b$, choose a subspace $W\, \subset\, H^0(A,\, TA)$
of dimension $b$ such  that the composition of maps
\begin{equation}\label{e6d}
A\times W \ \stackrel{\sigma}{\longrightarrow}\ TA \ \longrightarrow\ {\mathcal N}/{\mathcal N}_{m-1}
\end{equation}
is surjective over some point $x_0$ of $A$;
the homomorphism $\sigma$ in \eqref{e6d} is the evaluation
map that sends any $(x,\, w)\, \in\, A\times W$ to $w(x)\, \in\, T_xA$ (see \eqref{e1}).
Consequently, the composition of maps in \eqref{e6d} is surjective over a Zariski 
open subset $U$ of $A$ containing $x_0$. Let
\begin{equation}\label{e6c}
\varphi\ :\ {\mathcal W}\ :=\ A\times W\ \longrightarrow\ {\mathcal N}/{\mathcal N}_{m-1}
\end{equation}
be the composition of maps in \eqref{e6d}. It induces a homomorphism between the
determinant line bundles 
$$
\det\varphi\ :\ \det {\mathcal W}\,=\,
\bigwedge\nolimits^b {\mathcal W}\ \longrightarrow\ \det ({\mathcal N}/{\mathcal N}_{m-1})
$$
(see \cite[Ch.~V, \S~6]{Ko} for the construction of the determinant line bundle).
Since $\det\varphi$ does not vanish on the open subset $U$ over which $\varphi$ is surjective, it follows that
$$
\text{degree}(\det ({\mathcal N}/{\mathcal N}_{m-1})) - \text{degree}(\det {\mathcal W})\ =\
\text{degree}(\text{Divisor}(\det\varphi)).
$$
Therefore, we have $\text{degree}(\text{Divisor}(\det\varphi))\ =\
\text{degree}(\det ({\mathcal N}/{\mathcal N}_{m-1}))$, because ${\mathcal W}$ is a trivial vector bundle
(see \eqref{e6c}). Hence from \eqref{e6b} we have $\text{degree}(\text{Divisor}(\det \varphi))\,=\, 0$.
But this implies that $\text{Divisor}(\det \varphi)\,=\, 0$ (note that $\text{Divisor}(\det \varphi)$ is
an effective divisor). Since $\text{Divisor}(\det \varphi)\,=\, 0$, it follows that
$\varphi$ is an isomorphism. Consequently,
${\mathcal N}/{\mathcal N}_{m-1}$ is a trivial vector bundle.

As noted before, ${\mathcal N}/{\mathcal N}_{m-1}$ is a quotient of $TA$. Let
$$
{\mathcal K}\ \subset\ TA
$$
be the kernel of the quotient map $TA \, \longrightarrow\, {\mathcal N}/{\mathcal N}_{m-1}$, so
we have a short exact sequence
\begin{equation}\label{e7}
0\, \longrightarrow\, {\mathcal K}\, \longrightarrow\, TA\, \longrightarrow\,
{\mathcal N}/{\mathcal N}_{m-1}\, \longrightarrow\, 0.
\end{equation}
Note that we have
\begin{equation}\label{e8}
{\mathcal F}\ \subset\ {\mathcal K},
\end{equation}
because ${\mathcal N}/{\mathcal N}_{m-1}$ is a quotient of $(TA)/{\mathcal F}\,=\, \mathcal N$
(see \eqref{e4}). Since $\varphi$ is an isomorphism, we have
$$
TA \ = \ {\mathcal K}\oplus \sigma ({\mathcal W})
$$
(see \eqref{e6d} for $\sigma$). As $TA$ is a trivial vector bundle, and $\mathcal K$ is a direct summand of
it, we conclude that $\mathcal K$ is also a trivial vector bundle \cite[p.~315, Theorem 3]{At1}.

Since $\mathcal K$ is trivial, it follows that $H^0(A,\, {\mathcal K})$ generates $\mathcal K$
using the evaluation map (see \eqref{e1}).
Hence from \eqref{e8} it follows that $\mathcal F$ is contained in the subsheaf of $TA$ generated
by $H^0(A,\, {\mathcal K})\, \subset\, H^0(A,\, TA)$. Since $\mathcal F$ is generating, this implies
that ${\mathcal K}\,=\, TA$. Hence from \eqref{e7} it follows that ${\mathcal N}/{\mathcal N}_{m-1}\,
=\, 0$.

But this is a contradiction because ${\mathcal N}/{\mathcal N}_{m-1}\, \not=\, 0$. Since
\eqref{e6b} led to this contradiction, we conclude that \eqref{e6} holds.

{}From \eqref{e6} it follows that
$$
\mu_{\rm min}(\bigwedge\nolimits^j {\mathcal N})\ \geq \ j\cdot\mu_{\rm min}({\mathcal N})\ > \ 0
$$
for all $1\, \leq\, j\, \leq\, \text{rank}({\mathcal N}) \,=\, d-r$. Hence we have
\begin{equation}\label{e9}
\mu_{\rm max}(\bigwedge\nolimits^j {\mathcal N}^*)\ =\ -
\mu_{\rm min}(\bigwedge\nolimits^j {\mathcal N}) \ < \ 0
\end{equation}
for all $1\, \leq\, j\, \leq\, d-r$.

The given condition that $\mu_{\rm max}(V)\, \leq\, 0$ and \eqref{e9} combine together to imply
that
$$
\mu_{\rm max}(V\otimes \bigwedge\nolimits^j{\mathcal N}^*)\ =\ \mu_{\rm max}(V)+
\mu_{\rm max}(\bigwedge\nolimits^j {\mathcal N}^*) \ < \ 0
$$
for all $1\, \leq\, j\, \leq\, d-r$. From this it follows immediately that
$$
H^0(A,\, V\otimes \bigwedge\nolimits^j{\mathcal N}^*)\ =\ 0
$$
for all $j\, \geq\, 1$.
\end{proof}

\section{Turbulent foliation  and transversely branched Cartan geometry}\label{section foliation}

Branched Cartan geometries on foliated manifolds were introduced in \cite{BD}. In this
section we recall their definition.

\subsection{Partial connection along a foliation}

Assume that $\mathcal F$ in \eqref{e2} is closed under the operation of Lie bracket of vector fields,
or in other words, $\mathcal F$ is a holomorphic foliation on $A$.

A nonsingular holomorphic foliation $\mathcal F$ on a compact complex torus $A$ such that its holomorphic tangent bundle is 
a generating subbundle in $TA$ is called a {\it smooth turbulent foliation}.

It should be emphasized that while the above notion of turbulent foliation coincides with that of Ghys for codimension one 
smooth foliations, it is less general that the definition adopted in \cite{Br2, PS} where the authors work with possibly 
singular turbulent foliations.

Let $H$ be a complex Lie group. Its Lie algebra will be denoted by $\mathfrak h$. Let
\begin{equation}\label{g1}
p\ :\ E_H\ \longrightarrow\ A
\end{equation}
be a holomorphic principal $H$--bundle on $A$. Let
\begin{equation}\label{dp}
\mathrm{d}p\ :\ TE_H\ \longrightarrow\ p^*TA
\end{equation}
be the differential of the projection map $p$ in \eqref{g1}. The quotient
$$
\text{ad}(E_H)\ :=\ \text{kernel}(\mathrm{d}p)/H\ \longrightarrow\ A
$$
is the adjoint vector bundle for $E_H$. The quotient
$$
\text{At}(E_H)\ :=\ (TE_H)/H\ \longrightarrow\ A
$$
is known as the Atiyah bundle for $E_H$ \cite{At2}.
Consider the short exact sequence of holomorphic vector bundles on $E_H$
$$
0\, \longrightarrow\, \text{kernel}(\mathrm{d}p)\, \longrightarrow\,
\mathrm{T}E_H \, \stackrel{\mathrm{d}p}{\longrightarrow}\,p^*TA
\, \longrightarrow\, 0.
$$
Taking its quotient by $H$, we get the following short exact sequence of
vector bundles on $A$
\begin{equation}\label{at1}
0\, \longrightarrow\, \text{ad}(E_H)\, \xrightarrow{\,\,\,\iota''\,\,\,}\,\text{At}(E_H)\,
\xrightarrow{\,\,\,\widehat{\mathrm{d}}p\,\,\,}\, TA\, \longrightarrow\, 0,
\end{equation}
where $\widehat{\mathrm{d}}p$ is constructed from $\mathrm{d}p$. Now define the subbundle
\begin{equation}\label{atF}
\text{At}_{\mathcal F}(E_H)\, :=\, (\widehat{\mathrm{d}}p)^{-1}({\mathcal F})\,\subset\,
\text{At}(E_H).
\end{equation}
So from \eqref{at1} we get the short exact sequence
\begin{equation}\label{at2}
0\, \longrightarrow\, \text{ad}(E_H)\, \longrightarrow\,\text{At}_{\mathcal F}(E_H)\,
\xrightarrow{\,\,\,\mathrm{d}'p\,\,\,} \, {\mathcal F}\, \longrightarrow\, 0\, ,
\end{equation}
where $\mathrm{d}'p$ is the restriction of $\widehat{\mathrm{d}}p$
in \eqref{at1} to the subbundle $\text{At}_{\mathcal F}(E_H)$.

A partial holomorphic connection on $E_H$ in the direction of $\mathcal F$ is a holomorphic
homomorphism
$$
\theta\, :\, {\mathcal F}\, \longrightarrow\, \text{At}_{\mathcal F}(E_H)
$$
such that $\mathrm{d}'p\circ\theta\,=\, \text{Id}_{\mathcal F}$, where
$\mathrm{d}'p$ is the homomorphism in \eqref{at2}. Giving such a homomorphism
$\theta$ is equivalent to giving a homomorphism
$\varpi\, :\, \text{At}_{\mathcal F}(E_H)\, \longrightarrow\, \text{ad}(E_H)$ such that the
composition of maps
$$
\text{ad}(E_H) \,\hookrightarrow\,
\text{At}_{\mathcal F}(E_H)\, \stackrel{\varpi}{\longrightarrow}\, \text{ad}(E_H)
$$
coincides with the identity map of $\text{ad}(E_H)$, where the inclusion of $\text{ad}(E_H)$ in
$\text{At}_{\mathcal F}(E_H)$ is the injective homomorphism in \eqref{at2}.

Given a partial connection $\theta\, :\, {\mathcal F}\, \longrightarrow\, \text{At}_{\mathcal F}(E_H)$,
and any two locally defined holomorphic sections $s_1$ and $s_2$ of $\mathcal F$, consider
the locally defined section $\varpi ([\theta(s_1),\, \theta(s_2)])$ of $\text{ad}(E_H)$. This defines an
${\mathcal O}_A$--linear homomorphism
$$
{\mathcal K}(\theta) \, \in\, H^0(A,\, \text{Hom}(\bigwedge\nolimits^2{\mathcal F},\, \text{ad}(E_H)))
\,=\, H^0(A,\, \text{ad}(E_H)\otimes \bigwedge\nolimits^2{\mathcal F}^*),
$$
which is called the \textit{curvature} of the connection $\theta$. The connection $\theta$ is
called \textit{flat} (or \textit{integrable}) if ${\mathcal K}(\theta)$ vanishes identically.

The image of a  flat connection $\theta$ in the subbundle  $\text{At}_{\mathcal F}(E_H)$  is a $H$-invariant foliation lifting $\mathcal F$.

Recall that a (flat partial)  connection on the principal $H$-bundle   $E_H$  induces a canonical (flat partial) connection on  any associated bundle obtained through a representation of the structural group $H$. In particular, it induces a (flat partial)   
 connection on the adjoint bundle $\text{ad}(E_H)$.

For any partial connection $\theta$ on $E_H$, let
\begin{equation}\label{tp}
\theta'\ :\ {\mathcal F}\ \longrightarrow\ \text{At}(E_H)
\end{equation}
be the homomorphism given by the composition of maps $${\mathcal F}\, \stackrel{\theta}{\longrightarrow}\,
\text{At}_{\mathcal F}(E_H)\, \hookrightarrow\, \text{At}(E_H).$$
Note that from \eqref{at1} we have an exact sequence
\begin{equation}\label{at3}
0\, \longrightarrow\, \text{ad}(E_H)\, \stackrel{\iota'}{\longrightarrow}\, \text{At}(E_H)/\theta'({\mathcal F})
\, \xrightarrow{\,\,\,\widehat{\mathrm{d}}p\,\,\,}\, TA/{\mathcal F}\, =\, {\mathcal N}\, \longrightarrow\, 0\, ,
\end{equation}
where $\iota'$ is given by $\iota''$ in \eqref{at1}.

Consider the quotient ${\mathcal N}$ in \eqref{e4}. It is known that the normal bundle $\mathcal N$ to a holomorphic 
foliation $\mathcal F$ always admit a flat holomorphic partial connection ${\nabla}^{\mathcal F}$ in the direction of 
$\mathcal F$ called the {\it Bott connection}. For a local holomorphic section $V$ of $TA$ and a local holomorphic section 
$N$ of $\mathcal N$, the Bott connection operates as ${\nabla}^{\mathcal F}_VN\, =\,q(\lbrack V,\, \widetilde N\rbrack)$,
where $\widetilde N$ is a local holomorphic section of $TA$ lifting $N$ and $q$ is the quotient map defined
in \eqref{e4b}: this does not 
depend on the chosen local lift $\widetilde N$ and defines a flat holomorphic partial connection in the direction of $\mathcal 
F$ (see, for instance, \cite[Section 11.5]{BD2}).

\begin{lemma}[{\cite[p.~40, Lemma 2.1]{BD}}]\label{lem2}
Let $\theta$ be a flat partial connection on $E_H$. Then $\theta$ produces a flat partial
connection on ${\rm At}(E_H)/\theta'({\mathcal F})$ that satisfies the condition that the
homomorphisms in the exact sequence \eqref{at3} are partial  connection preserving (where  the adjoint bundle $\text{ad}(E_H)$ is endowed with the partial  connection induced from $(E_H, \theta)$ and $\mathcal N$ is equipped with the Bott connection ${\nabla}^{\mathcal F}$).
\end{lemma}

\subsection{Transversely branched Cartan geometry}

Let $G$ be a connected complex Lie group and $H\, \subset\, G$ a complex Lie subgroup.
The Lie algebra of $G$ will be denoted by $\mathfrak g$. As in \eqref{g1}, $E_H$
is a holomorphic principal $H$--bundle on $A$. Let
\begin{equation}\label{eg}
E_G\,=\, E_H\times^H G\,\longrightarrow\, A
\end{equation}
be the holomorphic principal $G$--bundle on $A$ obtained by extending the structure group of $E_H$ using
the inclusion of $H$ in $G$. The inclusion of $\mathfrak h$ in $\mathfrak g$ produces a fiber-wise
injective homomorphism of Lie algebra bundles
\begin{equation}\label{i1}
\iota\, :\, \text{ad}(E_H)\,\longrightarrow\,\text{ad}(E_G),
\end{equation}
where $\text{ad}(E_G)\,=\, E_G\times^G{\mathfrak g}$ is the adjoint bundle for $E_G$.
Let $\theta$ be a flat partial connection on $E_H$ in the direction of 
$\mathcal F$. So $\theta$ induces flat partial connections on the associated bundles $E_G$, 
$\text{ad}(E_H)$ and $\text{ad}(E_G)$.

A  {\it  transversely branched holomorphic Cartan geometry}  of type $(G,\, H)$
on the foliated torus $(A,\, {\mathcal F})$ is
\begin{itemize}
\item a holomorphic principal $H$--bundle $E_H$ on $A$ equipped with a flat partial
connection $\theta$ in the direction of $\mathcal F$, and

\item a holomorphic homomorphism
\begin{equation}\label{beta}
\beta\,:\, \text{At}(E_H)/\theta'({\mathcal F})\, \longrightarrow\,
\text{ad}(E_G),
\end{equation}
\end{itemize}
(see \eqref{tp} for $\theta'$) such that the following three conditions hold:
\begin{enumerate}
\item $\beta$ is partial connection preserving,

\item $\beta$ is an isomorphism over a nonempty open subset of $A$, and

\item the following diagram is commutative:
\begin{equation}\label{cg1}
\begin{matrix}
0 &\longrightarrow & \text{ad}(E_H) &\stackrel{\iota'}{\longrightarrow} &
\text{At}(E_H)/\theta'({\mathcal F}) &
\longrightarrow & {\mathcal N} &\longrightarrow & 0\\
&& \Vert &&~ \Big\downarrow\beta && ~ \Big\downarrow\overline{\beta}\\
0 &\longrightarrow & \text{ad}(E_H) &\stackrel{\iota}{\longrightarrow} &
\text{ad}(E_G) &\longrightarrow &
\text{ad}(E_G)/\text{ad}(E_H) &\longrightarrow & 0
\end{matrix}
\end{equation}
\end{enumerate}
where the top exact sequence is the one in \eqref{at3}, and $\iota$ is the homomorphism
in \eqref{i1}.

Notice that condition (2) in the above definition implies that the codimension of the foliation $\mathcal F$ coincides with 
the dimension of the homogeneous model space $G/H$.

Let $n$ be the complex dimension of $\mathfrak g$. Consider the homomorphism of $n$-th
exterior products
$$
\bigwedge\nolimits^n\beta\, :\, \bigwedge\nolimits^n(\text{At}(E_H)/\theta'({\mathcal F}))
\, \longrightarrow\, \bigwedge\nolimits^n\text{ad}(E_G)
$$
induced by $\beta$ in \eqref{beta}. The divisor $\text{Div}(\bigwedge\nolimits^n\beta)$
is called the \textit{branching divisor} for $((E_H,\, \theta),\, \beta)$.
We call $((E_H,\, \theta),\, \beta)$ a  {\it holomorphic
Cartan geometry}  if $\beta$ is an isomorphism over the entire $A$.

\subsection{Connection on $E_G$}

Let $((E_H,\, \theta),\, \beta)$ be a transversely branched Cartan geometry of type 
$(G,\, H)$ on the foliated manifold $(A,\, {\mathcal F})$. Consider the homomorphism
\begin{equation}\label{eh}
\text{ad}(E_H)\,\longrightarrow\, {\rm ad}(E_G)
\oplus \text{At}(E_H)\, , \ \ v \, \longmapsto\,
(\iota(v),\, -\iota''(v));
\end{equation}
see \eqref{i1} and \eqref{at1} for $\iota$ and $\iota''$ respectively.
The corresponding quotient $({\rm ad}(E_G)\oplus \text{At}(E_H))/\text{ad}(E_H)$
is identified with the Atiyah bundle
${\rm At}(E_G)$. Let
$$
\beta'\ :\ {\rm At}(E_H)\ \longrightarrow \ \text{ad}(E_G)
$$
be the composition of homomorphisms
$$
{\rm At}(E_H)\, \longrightarrow\, \text{At}(E_H)/\theta'({\mathcal F})
\,\stackrel{\beta}{\longrightarrow}\, \text{ad}(E_G),
$$
where the first homomorphism is the quotient map, and $\theta'$
(respectively, $\beta$) is the one in \eqref{tp} (respectively, \eqref{beta}). The homomorphism
\begin{equation}\label{hv}
{\rm ad}(E_G)\oplus \text{At}(E_H)\, \longrightarrow\, {\rm ad}(E_G)\, , \ \
(v,\, w) \, \longmapsto\, v+\beta'(w)
\end{equation}
vanishes on the image of $\text{ad}(E_H)$ by the map in \eqref{eh}. Therefore,
the homomorphism in \eqref{hv} produces a homomorphism
\begin{equation}\label{vp}
\varphi\, :\, \text{At}(E_G)\,=\, ({\rm ad}(E_G)\oplus \text{At}(E_H))/\text{ad}(E_H)
\,\longrightarrow\, \text{ad}(E_G)\, .
\end{equation}
The composition
$$
\text{ad}(E_G)\,\hookrightarrow\, \text{At}(E_G)\, \stackrel{\varphi}{\longrightarrow}\,\text{ad}(E_G)
$$
clearly coincides with the identity map of $\text{ad}(E_G)$. Consequently, $\varphi$ defines a holomorphic
connection on the principal $G$--bundle $E_G$ \cite{At2}. Let
${\rm Curv}(\varphi)\, \in\, H^0(A,\, \text{ad}(E_G)\otimes\Omega^2_A)$
be the curvature of the connection $\varphi$ on $E_G$.

\begin{lemma}[{\cite[p.~42, Lemma 3.1]{BD}}]\label{lem3}
The curvature ${\rm Curv}(\varphi)$ lies in the image of the homomorphism
$$
H^0(A,\, {\rm ad}(E_G)\otimes\bigwedge\nolimits^2 {\mathcal N}^*)
\, \hookrightarrow\, H^0(A,\, {\rm ad}(E_G)\otimes\Omega^2_A)
$$
given by the inclusion $q^*\, :\, {\mathcal N}^*\, \hookrightarrow\,
\Omega^1_A$ (the dual of the projection in \eqref{e4b}).
\end{lemma}

The transversely branched Cartan geometry $((E_H,\, \theta),\, \beta)$ will be called
\textit{flat} (or \textit{integrable}) if the curvature ${\rm Curv}(\varphi)$ vanishes identically.

Notice that for codimension one foliations, the normal bundle $\mathcal N$ has rank one and the above curvature tensor vanishes identically. Consequently, for codimension one foliations, any transversely branched holomorphic Cartan geometry is automatically flat.

A flat transversely holomorphic Cartan geometry with model $(G,\,H)$ induces on the foliated manifold $(A,\, \mathcal F)$ a {\it 
transversely holomorphic $G/H$--structure}. This means that $A$ admits an open cover by open sets $U_i$ equipped with 
holomorphic submersions $f_i \,:\, U_i \,\longrightarrow\, G/H$ such that the restriction of the foliation $\mathcal F$ on 
$U_i$ is given by the level sets of the map $f_i$. Moreover, for each connected component of the nonempty intersections $U_i 
\cap U_j$, there exists an element $g_{ij} \,\in\, G$ such that $f_i \,=\,g_{ij} \circ f_j$.

The above description holds also for flat {\it branched} holomorphic Cartan geometry, except that the local maps $f_i \,:\, U_i 
\,\longrightarrow\, G/H$ defining the foliation $\mathcal F$ in restriction to $U_i$ are submersions at the generic point (the 
maps $f_i$ are allowed to admit a branching divisor) (see \cite{BD,BD2}).

\section{Rigidity of branched Cartan geometries on torus} \label{section rigidity}

As before, $G$ is a connected complex Lie group, and $H\, \subset\, G$ is a complex Lie subgroup and $(A,\, \mathcal F)$
is a foliated compact complex torus. Let $E_H$ be a holomorphic principal $H$--bundle on $A$ equipped with
a flat partial connection $\theta$ in the direction of $\mathcal F$.

\begin{theorem}\label{thm1}\mbox{} Assume  $\mathcal F$ is  a smooth turbulent foliation on the  compact complex torus $A$. 
\begin{enumerate}
\item There can be at most one transversely branched  holomorphic Cartan geometry of type $(G,\, H)$ on the foliated torus  $(A, \mathcal F)$ admitting $(E_H, \theta)$ as underlying  $H$-principal bundle equipped  with the flat   partial  connection $\theta$  along $\mathcal F$.

\item Any transversely branched holomorphic  Cartan geometry of type $(G,\, H)$ on the foliated torus $(A, \mathcal F)$  is flat.
\end{enumerate}
\end{theorem}

\begin{proof} 

To prove (1), let $\beta_1$ and $\beta_2$ be two transversely branched Cartan geometries of type $(G,\, H)$
on $(E_H,\, \theta)$. Consider the principal $G$--bundle $E_G$ in \eqref{eg}. Let $\varphi_1$ (respectively,
$\varphi_2$) be the holomorphic connections on $E_G$ produced by $\beta_1$ (respectively,
$\beta_2$); see \eqref{vp} for the constructions of $\varphi_1$ and $\varphi_2$. Since the space of
holomorphic connections on $E_G$ is an affine space for $H^0(A,\, \text{ad}(E_G)\otimes\Omega^1_A)$, we
have
\begin{equation}\label{e10}
\beta_1-\beta_2\ \in\ H^0(A,\, \text{ad}(E_G)\otimes\Omega^1_A).
\end{equation}

Let $\varphi'_1$ (respectively, $\varphi'_2$) be the partial connection on $E_G$
given by $\varphi_1$ (respectively, $\varphi_2$). On the other hand, the partial connection
$\theta$ on $E_H$ induces a partial connection on $E_G$, because $E_G$ is the principal $G$--bundle
on $A$ obtained by extending the structure group of $E_H$ using the inclusion map $H\, \hookrightarrow
\, G$. Let $\widetilde{\theta}$ be the partial connection on $E_G$ induced by $\theta$.
Note that $\varphi'_1$ and $\varphi'_2$ coincide with $\widetilde{\theta}$. Consequently, from
\eqref{e10} we have
\begin{equation}\label{e11}
\beta_1-\beta_2\ \in\ H^0(A,\, \text{ad}(E_G)\otimes {\mathcal N}^*);
\end{equation}
note that the homomorphism $${\rm Id}_{\text{ad}(E_G)}\otimes q^*\ :\ 
\text{ad}(E_G)\otimes {\mathcal N}^*\ \longrightarrow\
\text{ad}(E_G)\otimes\Omega^1_A,$$ where $q^*$ is the dual of the homomorphism in \eqref{e4b},
produces an injective homomorphism
$$
H^0(A,\, \text{ad}(E_G)\otimes {\mathcal N}^*)\ \longrightarrow\ 
H^0(A,\, \text{ad}(E_G)\otimes\Omega^1_A).
$$

Since $\text{ad}(E_G)$ is a holomorphic vector bundle, admitting a holomorphic connection, on a
compact complex torus, we know that $\text{ad}(E_G)$ is semistable of degree zero \cite[p.~41, Theorem 4.1]{BG}.
Hence from Lemma \ref{lem1} we know that
$$
H^0(A,\, \text{ad}(E_G)\otimes {\mathcal N}^*)\ = \ 0.$$
Now from \eqref{e11} it follows that $\beta_1\,=\, \beta_1$. This proves (1).

To prove (2), recall from Lemma \ref{lem3} that the curvature of
a transversely branched Cartan geometry of type $(G,\, H)$ on $(E_H,\, \theta)$ is a
holomorphic section of ${\rm ad}(E_G)\otimes\bigwedge\nolimits^2 {\mathcal N}^*$. Again from
Lemma \ref{lem1} we know that
$$
H^0(A,\, {\rm ad}(E_G)\otimes\bigwedge\nolimits^2 {\mathcal N}^*)\ = \ 0.$$
Hence any transversely branched Cartan geometry of type $(G,\, H)$ on $(E_H,\, \theta)$
is flat.
\end{proof}



\begin{thebibliography}{ZZZZ}

\bibitem[At1]{At1} M. F. Atiyah, On the Krull-Schmidt theorem with application to sheaves,
{\it Bull. Soc. Math. Fr.} {\bf 84} (1956), 307--317.

\bibitem[At2]{At2} M. F. Atiyah, Complex analytic connections in fibre
bundles, \textit{Trans. Amer. Math. Soc.} \textbf{85} (1957), 181--207.

\bibitem[BD]{BD} I. Biswas and S. Dumitrescu, Transversely holomorphic
branched Cartan geometry, {\it Jour. Geom. Phy.} {\bf 134} (2018), 38--47.

\bibitem[BD2]{BD2} I. Biswas and S. Dumitrescu, Holomorphic $G$-structures and foliated Cartan geometries on  compact complex
manifolds, {\it Suveys in Geometry I}, Editor Athanase Papadopoulos, Springer Nature Switzerland AG, 
Chapter 11, 417-461, (2022).

\bibitem[BD3]{BD3}  I. Biswas and S. Dumitrescu, Holomorphic foliations with no transversely projective structure, (2024), to appear in {\it Ann. Fac. Sci. Toulouse Math.}

\bibitem[BG]{BG} I. Biswas and T. L. G\'omez, Connections and Higgs fields on
a principal bundle, {\it Ann. Global Anal. Geom.} {\bf 33} (2008), 19--46.

\bibitem[Br1]{Br1} M. Brunella, Feuilletages holomorphes sur les surfaces complexes compactes, {\it Ann. Sci. 
\'Ec. Norm. Sup.} {\bf 30} (1997), 569--594.

\bibitem[Br2]{Br2} M. Brunella, Codimension one foliations on complex tori, {\it Ann. Fac. Sci. Toulouse Math.} 
{\bf 19} (2010), 405--418.

\bibitem[FMP]{FMP}  G. Fazoli, C. Melo and J.V. Pereira, Transversely projective structures on smooth foliations on surfaces, {\it Manuscripta Math.}, {\bf 176}, (2025).

\bibitem[Gh]{Gh} E. Ghys, Codimension one foliations on complex tori, {\it Ann. Fac. Sci. Toulouse Math.}
{\bf 5} (1996), 493--519.

\bibitem[HL]{HL} D. Huybrechts and M. Lehn, {\it The geometry of moduli spaces of sheaves}, Aspects
of Mathematics, E31, Friedr. Vieweg~\&~Sohn, Braunschweig, 1997.

\bibitem[Ko]{Ko} S. Kobayashi, \textit{Differential geometry of complex vector bundles}, Princeton 
University Press, Princeton, NJ, Iwanami Shoten, Tokyo, 1987.

\bibitem[LBP]{LoP} F. Lo Bianco and J. V. Pereira, Smooth foliations
on homogeneous compact K\"ahler manifolds, {\it. Ann. Fac. Sci. Toulouse Math.} {\bf 25} (2016), 141--159.

\bibitem[LP]{LP} F. Loray and J. V. Pereira, Transversely projective foliations on surfaces: existence
of minimal form and prescription of monodromy, {\it Internat. Jour. Math.} {\bf 18} (2007), 723--74.

\bibitem[LPT]{LPT} F. Loray, J. V. Pereira and F. Touzet, Representations of quasi-projective groups, flat 
connections and transversely projective foliations, {\it J. \'Ec. Polytech.} {\bf 3} (2006), 263--308.

\bibitem[PS]{PS}  I. Pan and M. Sebastiani, Classification des feuilletages turbulents, {\it Ann. Fac. Sci. Toulouse Math.} {\bf 3} (2003), 395--413.

\bibitem[Sc]{Sc} B. A. Sc\'ardua, Transversely affine and transversely projective holomorphic
foliations, {\it Ann. Sci. \'Ecole Norm. Sup.} {\bf 30} (1997), 169--204.

\bibitem[Se]{Se} B. Seke, Sur les structures transversalement affines des feuilletages de codimension un,
{\it Ann. Inst. Fourier} {\bf 30} (1980), 1--29. 

\bibitem[Sh]{Sh} R. Sharpe, {\it Differential Geometry: Cartan's generalization of Klein's Erlangen 
program}, Graduate Texts in Mathematics, Vol. \textbf{166}, Springer-Verlag, New York, 1997.

\end{thebibliography}
\end{document}